\documentclass[12pt, reqno]{amsart}
\usepackage{amssymb}
\usepackage{amsmath}
\usepackage[utf8]{inputenc}
\usepackage{longtable}
\usepackage{cite}
\usepackage[all,cmtip]{xy}
\usepackage{amsmath}
\usepackage{amsfonts}
\usepackage{amssymb}
\usepackage{amsthm}
\usepackage{mathtools}
\usepackage{mathrsfs}
\usepackage{float}
      \theoremstyle{definition}
\newtheorem{defi}{Definition}[section]
\theoremstyle{plain}
\newtheorem{thm}[defi]{Theorem}
\newtheorem{prop}[defi]{Proposition}
\newtheorem{cor}[defi]{Corollary}
\newtheorem{lemma}[defi]{Lemma}
\theoremstyle{remark}
\newtheorem{question}[defi]{Question}
\newtheorem{problem}[defi]{Problem}

\newtheorem{ex}[defi]{Example}

\theoremstyle{definition}

\newtheorem*{ack}{Acknowledgement}

\newcommand{\lra}{\leftrightarrow}

\newcommand{\longlra}{\longleftrightarrow}

\newcommand{\longra}{\longrightarrow}

\newcommand{\ke}{{\mathcal E}}

\newcommand{\ko}{{\mathcal O}}
\newcommand{\kp}{{\mathcal P}}
\newcommand{\kq}{{\mathcal Q}}

\newcommand{\C}{{\mathbb C}}

\newcommand{\Q}{{\mathbb Q}}

\newcommand{\Z}{{\mathbb Z}}

\newcommand{\gm}{{\mathfrak m}}

\newcommand{\gp}{{\mathfrak p}}

\newcommand{\rH}{{\rm H}}
\newcommand{\rT}{{\rm T}}

\newcommand{\Aut}{\operatorname{Aut}}

\newcommand{\End}{\operatorname{End}}
\newcommand{\GL}{\operatorname{GL}}
\newcommand{\SL}{\operatorname{SL}}

\newcommand{\Gal}{\operatorname{Gal}}

\newcommand{\Hom}{\operatorname{Hom}}

\newcommand{\NS}{\operatorname{NS}}

\newcommand{\Div}{\operatorname{Div}}

\newcommand{\rank}{\operatorname{rank}}   
\newcommand{\disc}{\operatorname{disc}}
\newcommand{\Stab}{\operatorname{Stab}}
\newcommand{\Spl}{\operatorname{Spl}}
\newcommand{\lcm}{\operatorname{lcm}}

  \makeatletter
\newcommand{\xdashrightarrow}[2][]{\ext@arrow 0359\rightarrowfill@@{#1}{#2}}
\newcommand{\xdashleftarrow}[2][]{\ext@arrow 3095\leftarrowfill@@{#1}{#2}}
\newcommand{\xdashleftrightarrow}[2][]{\ext@arrow 3359\leftrightarrowfill@@{#1}{#2}}
\def\rightarrowfill@@{\arrowfill@@\relax\relbar\rightarrow}
\def\leftarrowfill@@{\arrowfill@@\leftarrow\relbar\relax}
\def\leftrightarrowfill@@{\arrowfill@@\leftarrow\relbar\rightarrow}
\def\arrowfill@@#1#2#3#4{%
  $\m@th\thickmuskip0mu\medmuskip\thickmuskip\thinmuskip\thickmuskip
   \relax#4#1
   \xleaders\hbox{$#4#2$}\hfill
   #3$%
}
\makeatother

	\title{Decompositions of singular abelian surfaces}
	\author{Roberto Laface}
	\address{Institut f\"{u}r Algebraische Geometrie, Gottfried Leibniz Universit\"{a}t Hannover, Welfengarten 1 30167 Hannover (Germany)}
	\email{laface@math.uni-hannover.de}
	\dedicatory{Dedicated to my father on the occasion of his 50th birthday.}

\begin{document}

\maketitle
\thispagestyle{empty}

\begin{abstract}
Given an abelian surface, the number of its distinct decompositions into a product of elliptic curves has been described by Ma. Moreover, Ma himself classified the possible decompositions for abelian surfaces of Picard number $1 \leq \rho \leq 3$. We explicitly find all such decompositions in the case of abelian surfaces of Picard number $\rho =4$. This is done by computing the transcendental lattice of products of isogenous elliptic curves with complex multiplication, generalizing a technique of Shioda and Mitani, and by studying the action of a certain class group on the factors of a given decomposition. We also provide an alternative and simpler proof of Ma's formula, and an application to singular K3 surfaces.
\end{abstract}

\setcounter{tocdepth}{1}
%level -1: part, 0: chapter, 1: section, etc.
\tableofcontents

\section{Introduction}
\noindent
After the ground breaking work of Shioda and Mitani \cite{shioda-mitani74}, and of Shioda and Inose \cite{shioda-inose77}, singular abelian surfaces, i.e.~abelian surfaces of maximum Picard number, have played a key role in the theory of K3 surfaces, because of the rich arithmetic information they carry. This data is encoded in the transcendental lattice, and it naturally transfers to singular K3 surfaces by means of a Shioda-Inose structure \cite{morrison84}. The associated singular K3 surface, which has the same transcendental lattice by a result of Shioda and Inose \cite{shioda-inose77}, inherits some of the arithmetic structure of the associated singular abelian surface. This has been employed, for instance, in the study of the field of definition of singular K3 surfaces by Sch\"{u}tt in \cite{schuett07}.\newline

\noindent
In \cite{ma11}, Ma gives a formula for the number of decompositions of an abelian surface into the product of elliptic curves; this expression is in terms of the arithmetic of the transcendental lattice. The proof builds on lattice theoretical methods, and it works for abelian surfaces of any Picard number. Also, he is able to classify all the distinct decompositions of a given abelian surface of Picard number $\rho\leq 3$. However, there is no mention of the possible decompositions that can appear in the case of singular abelian surfaces. The main purpose of this paper is to classify the possible decompositions into the product of two elliptic curves that a singular abelian surface can admit.\\

\noindent 
The present paper consists of two parts: in the first one, we develop techniques that allow us to understand the behavior of and to compute the transcendental lattices of products of two CM elliptic curves which are mutually isogenous. We would like to stress that such techniques are of interest on their own, as they provide generalizations of previous results of Shioda and Mitani \cite{shioda-mitani74} about the geometry of abelian surfaces, and also of Gau\ss~ and Dirichlet in the theory of quadratic forms. The second part of this article is concerned with the problem of classifying all the possible decompositions of a given singular abelian surface, and it is the real motivation behind our studies. In doing so, we have tried to highlight the connection between the geometry of this class of surfaces and the arithmetic of quadratic forms as much as possible.\\

\noindent
The starting point of our study is the computation of the transcendental lattice of certain singular abelian surfaces. This will eventually suggests that we look at more general singular abelian surfaces, and, in order to do so, we will need to introduce and study the basic properties of the \textit{generalized Dirichlet composition}, a notion that generalizes the usual Dirichlet composition of quadratic forms. This notion is crucial for fully understanding how to compute transcendental lattices of arbitrary products of two CM elliptic curves which are mutually isogenous. We remark that this extends previous work of Shioda and Mitani \cite{shioda-mitani74}, where the authors computed the transcendental lattice of very special models of singular abelian surfaces in order to prove the surjectivity of the period map. The conclusion is that computing the transcendental lattice of a singular abelian surface boils down to composing two appropriate quadratic forms by means of the generalized Dirichlet composition (Proposition \ref{thm2}).
\\

\noindent
Afterwards, we turn to the study of the possible decompositions of a given singular abelian surface $A$. We distinguish two cases, according to whether the CM field $K$ of $A$ is one among $\Q(i)$ and $\Q(\sqrt{-3})$, or not. In the latter case, we are able to show that a certain class group acts on the set of decompositions of $A$, and that this action delivers all possible decompositions of $A$ (Theorem \ref{mainthm}). We also give a new proof of Ma's formula for the number of possible decompositions. The cases where $K=\Q(i)$ or $K= \Q(\sqrt{-3})$ are handled separately, according to the number of units in $\ko_{K}$. In both cases, we give a complete classification of the possible decompositions, and we also provide a formula for their number, again distancing ourselves from Ma's approach.\\

\noindent
The paper is organized as follows: in Section \ref{preliminaries} we go over all the necessary notions and basic result we need, and afterwards (Section \ref{firstdec}) we state Ma's result and make a couple of motivating remarks for what is studied thereinafter. In Section \ref{transcendental}, we study composition of forms living in different class groups, and we compute explicitly the transcendental lattice of a product abelian surface of Picard rank 4, getting even more candidate decompositions. Finally, Section \ref{alldec} deals with the problem of distinct decompositions: we show that we have build enough decompositions to match Ma's formula in most cases, and this incidentally leads to a new proof of Ma's formula. Afterwards, in Section \ref{remainingcases}, we completely solve the classification problem for decompositions in the remaining cases, also providing a new approach to the formula giving their number. We conclude the paper with an application to the field of moduli of singular K3 surfaces, and some open problems, in the hope that they might stimulate future research in this direction.

\begin{ack}
I would like to express my gratitude to my advisor Matthias Sch\"{u}tt for fruitful conversations on the topic, and for a key remark that led to the completion of the paper. I also thank Victor Gonz\'alez Alonso and Daniel Loughran for detailed comments and suggestions on earlier drafts of the manuscript.
\end{ack}

\section{Preliminaries}\label{preliminaries}

\subsection{Singular abelian surfaces}

We introduce briefly the basic theory of singular abelian surfaces; for a reference, the reader may see \cite{shioda-mitani74}. We will be
working over the field of complex numbers. If $X$ is a smooth algebraic surface, we can define the N\'{e}ron-Severi lattice of $X$:
it is the group of divisors on $X$, modulo algebraic equivalence, namely
\[\NS(X) := \Div(X) / \sim_{\rm alg},\]
together with restriction of the intersection form on $\rH^2(X,\Z)$. Its rank $\rho(X):=\rank \NS(X)$ is called \textit{Picard number} of $X$; the Picard number measures how many different
curves lie on a surface. By the Lefschetz theorem on $(1,1)$-classes, we have the bound
\[\rho(X) \leq h^{1,1}(X) = b_2(X) - 2p_g(X),\]
where $b_2(X) := \rank \rH^2(X,\Z)$ and $p_g(X) : =\dim_\C \rH^0(X,\omega_X)$.\newline

\noindent
We can consider the lattice 
\[ \rH^2(X,\Z)_\text{free} := \rH^2(X,\Z)/(\text{torsion}),\]
and since $\NS(X) \subset \rH^2(X,\Z)$, also $\NS(X)_\text{free} \subset \rH^2(X,\Z)_\text{free}$; the lattice $\NS(X)_\text{free}$ has signature $(1,\rho(X)-1)$. Its orthogonal complement
$\rT(X) \subset \rH^2(X,\Z)_\text{free}$ is called the \textit{transcendental lattice} of $X$, and it has signature 
\[(2p_g(X), h^{1,1}(X) - \rho(X)).\]
A smooth algebraic surface with maximum Picard number, i.e. $\rho(X) = h^{1,1}(X)$, is called a \emph{singular surface}. In this case, the
transcendental lattice acquires the structure of a positive definite lattice of rank $2p_g(X)$. Throughout the paper, we are going to consider
a special class of surfaces, namely singular abelian surfaces. Letting $A$ be such a surface, $\rho(A) = 4$ and $\rT(A)$ is a positive definite
integral binary form.\newline

\noindent
We now recall the structure of the period map of an abelian surface $A$ (not necessarily singular). The exponential sequence
\[0 \longra \Z \longra \ko_A \longra \ko_A^\times \longra 0\]
yields a long exact sequence in cohomology, from which we can extract a map
\[p_A \, : \, \rH^2(A,\Z) \longra \rH^2(A,\ko_A) \cong \C,\]
since $p_g(A)=1$; the map $p_A$ is called \textit{the period map of $A$}. By using the structure of complex torus of $A$, we make this more
explicit: since 
\[\rH^2(A,\Z) \cong \bigwedge^2 \rH^1(A,\Z) \qquad \text{and} \qquad \rH^1(A,\Z) = \rH_1(A,\Z)^\vee,\]
we can take a basis
$\lbrace v_1, v_2,v_3, v_4 \rbrace$ of $\rH_1(A,\Z)$ and the corresponding dual basis $\lbrace u^1,u^2,u^3,u^4 \rbrace$.
Then, setting $u^{ij}:=u^i \wedge u^j$, we get a basis of $\rH^2(A,\Z)$ by considering 
\[\lbrace u^{ij} \, \vert \, 1 \leq i < j \leq 4 \rbrace,\]
which also gives a basis of $\rH^2(A,\C)$. As an element of $\rH^2(A,\C) \cong \Hom(\rH^2(A,\Z),\C)$ (here we tacitly use Poincar\'e duality), the period map has the following description:
\[p_A = \sum_{i<j} \det(v_i \vert v_j) u^{ij},\]
where the notation $(v_i \vert v_j)$ indicates the matrix whose columns are $v_i$ and $v_j$. Notice that, since $\NS(A) = \ker(p_A)$ and $\rT(A) = (\ker(p_A))^\perp$, this allows us to explicit compute the N\'{e}ron-Severi and the
transcendental lattices. Also, the period map satisfies the period relations 
\[(p_A,p_A)=0 \qquad \text{and} \qquad (p_A,\overline{p_A})>0.\]

\subsection{Class group theory}

We recall a few facts on integral binary quadratic forms; for a detailed account, the reader is suggested to see \cite{cox13}. 
Given a form 
\[Q(x,y)=ax^2+bxy+cy^2\]
the quantity $\gcd(a,b,c)$ is called \textit{index of primitivity} and $Q$ is said \emph{primitive} if $\gcd(a,b,c)=1$. Sometimes, it is convenient to extract the \textit{primitive part} of a form $Q$: this is the quadratic form $Q_0$ such that $mQ_0 =Q$, $m$ being the index of primitivity of $Q$. A form $Q$ \emph{represents} $m \in \Z$ if $m=Q(x,y)$ for some
$x,y \in \Z$; if moreover $\gcd(x,y)=1$, then we say that $Q$ \emph{properly represents} $m \in \Z$. A quadratic form $Q$ as above
will be denoted in short by $Q=(a,b,c)$. Two forms $Q = (a,b,c)$ and $Q'=(a',b',c')$ are equivalent (properly equivalent, respectively) if there exists
$\begin{pmatrix} p & q \\ r & s\end{pmatrix} \in \GL_2(\Z)$ ($\SL_2(\Z)$, respectively) such that 
\[Q(px+qy,rx+sy) = Q'(x,y).\]
The following basic results give a hint on why it is important to know which numbers a form represents.
\begin{lemma}[Lemma 2.3 in \cite{cox13}]
A form $Q$ properly represents $m \in \Z$ if and only if $Q$ is properly equivalent to the form $(m,B,C)$, for some $B,C \in \Z$.
\end{lemma}

\begin{lemma}[Lemma 2.25 in \cite{cox13}]
Given a form $Q$ and an integer $M$, $Q$ represents infinitely many numbers prime to $M$.
\end{lemma}

\noindent
The \textit{discriminant} of a form $Q=(a,b,c)$ is the integer $D:=b^2-4ac$. The set of proper equivalence classes of primitive forms of discriminant $D$ is called the \emph{(form) class group of discriminant $D$},
and it is denoted by $C(D)$; we will denote the class of a form $Q$ by $[Q]$. The class group is equipped with the Dirichlet composition
of forms: by \cite[Lemma 3.2]{cox13}, if $Q=(a,b,c)$ and $Q'=(a',b',c')$ are primitive forms of discriminant $D$, such that
\[\gcd\Big(a,a',\frac{b+b'}{2}\Big)=1,\]
then the composition $Q*Q'$ is the form $(aa',B,C)$, where $C=\frac{B^2-D}{4aa'}$ and $B$ is the integer,
unique modulo $2aa'$, such that
\[\left.
\begin{cases}
B \equiv b \mod 2a, \\
B \equiv b' \mod 2a', \\
B^2 \equiv D \mod 4aa'.
\end{cases}\right.\]

\noindent
Naturally, we put $[Q]*[Q']:=[Q*Q']$. Notice that the arithmetic properties of $B$ allow us to rewrite $\tau(Q)$ and $\tau(Q')$:
indeed, we can always assume that
\[\tau(Q) = \frac{-B+\sqrt{D}}{2a}, \qquad \tau(Q')= \frac{-B+\sqrt{D}}{2a'},\]
and furthermore $\gcd(a,a',\frac{b+b'}{2})=1$.\newline

\noindent
Recall that, fixed a quadratic imaginary field $K$, an order $\ko$ is a subring of $K$ containing the unity of $K$ which has also the 
structure of a rank-two free $\Z$-module. Every order $\ko$ can be written in a unique way as
\[\ko=\Z + fw_K \Z, \quad w_K:=\frac{d_K+\sqrt{d_K}}{2}, \quad d_K:=\disc \ko_K, \quad f \in \Z^+.\]
The integer $f$ is called the \emph{conductor} of $\ko$, and it characterizes $\ko$ in a unique way; we will denote the order of conductor $f$ in $\ko_K$ by $\ko_{K,f}$. Similarly, a \emph{module} $M$ in $K$ is a rank-two $\Z$-submodule of $K$ (no condition on the unity). Two modules $M_1$ and $M_2$ are equivalent ($M_1 \sim M_2$) if they are homothetic,
i.e.~there exists $\lambda \in K$ such that $\lambda M_1 = M_2$. To any module $M$, we can associate its \emph{complex multiplication}
(CM) \emph{ring}\footnote{The name CM ring refers to the property $\ko_{M} = \End(\C/M)$.}
\[\ko_M := \lbrace x \in K \ \vert \ x M \subseteq M \rbrace.\]
Notice that $\ko_M$ is an order in $K$, and that equivalent modules in $K$ have the same CM ring. The product module $M_1 M_2$ is
defined in a natural way, and if $f_i$ is the conductor of $M_i$ ($i=1,2$), then $\ko_{M_1 M_2} = \ko_{K,(f_1, f_2)}$, the latter
being the order of conductor $(f_1,f_2)$ in $K$.\newline

\noindent
For an order $\ko$ in a quadratic field $K$, it is possible to define a class group $C(\ko)$: letting $I(\ko)$ denote the group of proper
fractional ideals, meaning those whose CM ring is $\ko$ itself, and letting $P(\ko)$ be the subgroup generated by the principal ones, we
set $C(\ko) := I(\ko)/P(\ko)$, and we call it the \emph{ideal class group} of $\ko$. An important result in algebraic number theory states
that if $\disc \ko = D$, then $C(D) \cong C(\ko)$. From now on, we will use interchangeably the two class groups to our convenience. The order of the class group $C(\ko)$ is called the \emph{class number} of $\ko$, and it is denoted by $h(\ko_{K,f})$, or $h(D)$ according to the isomorphism between form and ideal class group. There is a beautiful formula that describes the order of the class group of an order in terms of its conductor and the maximal order that contains it.

\begin{thm}[Theorem 7.24 in \cite{cox13}]\label{clnumbfor}
Let $\ko_{K,f}$ be the order of conductor $f$ in $\ko_K$. Then
\[h(\ko_{K,f}) = \frac{h(\ko_{K}) \cdot f}{[\ko_K^\times : \ko_{K,f}^\times]} \prod_{p \vert f} \Bigg(1- \Bigg(\dfrac{d_K}{p}\Bigg)\dfrac{1}{p} \Bigg),\]
where $p$ runs over the primes dividing the conductor $f$.
\end{thm}

\subsection{The moduli space of singular abelian surfaces}
Let $\Sigma^{\rm Ab}$ be the moduli space of singular abelian surfaces\footnote{By moduli space, we do not mean the solution to a moduli problem, but rather the set of isomorphism classes of singular abelian surfaces}. In \cite{shioda-mitani74}, Shioda and Mitani described $\Sigma^{\rm Ab}$ by means of the transcendental lattice $\rT(A)$ associated to any singular abelian surface $A$. We say that $\rT(A)$ is \textit{positively oriented} if 
\[\rT(A) = \Z \langle t_1, t_2 \rangle \qquad \text{and} \qquad \text{Im} ( p_A(t_1) / p_A(t_2) ) > 0.\]

\noindent
Notice that the transcendental lattice $\rT(A)$ is an even lattice $\begin{pmatrix} 2a & b \\ b & 2c \end{pmatrix}$, and thus we can always associate to it the quadratic form $(a,b,c)$. This realizes a 1:1 correspondence, and therefore we can naturally see the transcendental lattice as an integral binary quadratic form. We can associate to any quadratic form $Q=(a,b,c)$ an abelian surface $A_Q$. In order to describe the correspondence, we set
\[\tau \equiv \tau(Q):=\dfrac{-b+\sqrt{D}}{2a}, \qquad D:= \disc Q = b^2-4ac,\]
and we will denote by $E_{\tau}$ the elliptic curve $\C / \Lambda_{\tau}$, $\Lambda_{\tau}$ being the lattice $\Z + \tau\Z$.
The abelian surface associated to a form $Q$ is then defined as the product surface
\[A_Q:=E_{\tau} \times E_{a\tau+b}.\]
The mapping $Q \mapsto A_Q$ realizes a 1:1 correspondence between $\SL_2(\Z)$-conjugacy classes of binary forms and isomorphism classes
of singular abelian surfaces, namely
\[ \Sigma^{\rm Ab} \longlra \kq^+ / \SL_2(\Z),\]
$\kq^+$ being the set of positive definite integral binary quadratic forms. By dropping the orientation, we get a 2:1 map
$\Sigma^{\rm Ab} \longra \kq^+/\GL_2(\Z)$, which is just taking the transcendental lattice of an abelian surface:
\[\Sigma^{\rm Ab} \ni [A] \longmapsto [\rT(A)] \in \GL_2(\Z).\]
As a consequence, we get that every singular abelian surface $A$ is isomorphic to the product of two isogenous elliptic curves with
complex multiplication. Therefore, we can ask the following
\begin{question}\noindent
\begin{itemize}
 \item[1.] Given a singular abelian surface $A$, how many distinct decompositions of $A$ into the product of two elliptic curves are there?
 \item[2.] Can we list all the possible decompositions?
\end{itemize}
\end{question}

\noindent
Part 1 has been completely solved by Ma \cite{ma11}, for abelian surfaces of any Picard number. Concerning Part 2, in case $\rT(A) = Q$,
for a primitive form $Q$, the answer is given in \cite[Theorem 4.7]{shioda-mitani74}, and the formula depends on the structure of the
class group of a certain order. We don't discuss this any further here, because, for our purposes, we will need a different, and somehow easier,
interpretation of this result, which will be given in Lemma \ref{cor_ma2}.

\subsection{Class field theory}

For later reference, we need to state a couple of facts from class field theory; see \cite{cox13} for an account on the subject. Let $K$ be
a number field, and let $\gm$ be a \emph{modulus} in $K$, i.e.~a formal product
\[\gm = \prod_\gp \gp^{n_\gp}\]
over all primes $\gp$ of $K$, finite or infinite, where the exponents satisfy
\begin{enumerate}
 \item $n_\gp \geq 0$, and at most finitely many are nonzero;
 \item $n_{\gp} = 0$, for $\gp$ a complex infinite prime;
 \item $n_\gp \leq 1$, for $\gp$ a real infinite prime.
\end{enumerate}
Consequently, any modulus $\gm$ can be written as $\gm = \gm_0 \gm_\infty$, where $\gm_0$ is an $\ko_K$-ideal and $\gm_\infty$ is a product
of distinct real infinite primes of $K$. We define $I_K(\gm)$ to be the group of fractional ideals of $K$ that are coprime to $\gm$, and we
let $P_{K,1}(\gm)$ be the subgroup of $I_K(\gm)$ generated by the principal ideals $\alpha \ko_K$, where $\alpha \in \ko_K$ satisfies
\[\alpha \equiv 1 \mod \gm_0, \, \sigma(\alpha) >0 \, \text{for every real infinite prime } \sigma \vert \gm_\infty .\]
One sees that $P_{K,1}(\gm)$ is of finite index in $I_K(\gm)$. A subgroup $H \subseteq I_K(\gm)$ is called a \emph{congruence subgroup} for
$\gm$ if
\[P_{K,1}(\gm) \subseteq H \subseteq I_K(\gm),\]
and the quotient $I_k(\gm) / H$ is called a \emph{generalized class group} of $\gm$. Let now $L$ be an abelian extension of $K$, and assume
that $\gm$ is divisible by all primes of $K$ that ramify in $L$. Then, for a given prime $\gp$ in $K$, one can define the Frobenius element
associated to $\gp$ by means of the Artin symbol $\big( \frac{L/K}{\gp} \big) \in \Gal(L/K)$, thus defining a map
\[\Phi^{L/K}_\gm : I_K(\gm) \longra \Gal(L/K),\]
called the \emph{Artin map} for $L/K$ and $\gm$.\newline

\noindent
Suppose we have the diagrams of orders

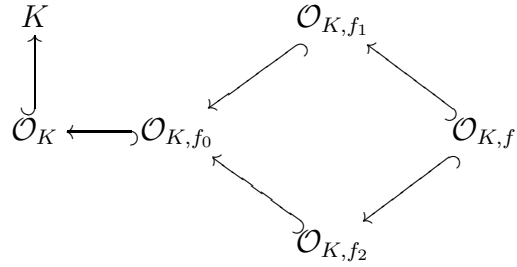
\[\xymatrix{
K & &\ko_{K,f_1} \ar@{^{(}->}[ld] & \\
\ko_K \ar@{^{(}->}[u] & \ko_{K,f_0} \ar@{^{(}->}[l] & & \ko_{K,f} \ar@{^{(}->}[lu] \ar@{^{(}->}[ld] \\
 & &\ko_{K,f_2} \ar@{^{(}->}[lu] &
}
\]

\noindent
where $f_1,f_2 \geq 1$, $f_0= \gcd(f_1,f_2)$ and $f=\text{lcm}(f_1,f_2)$. Let $ i = 0,1,2,\emptyset$; since
\[P_{K,1}(f_i \ko_K) \subseteq P_{K,\Z}(f_i) \subseteq I_K(f_i) = I_K(f_i \ko_K),\]
by the Existence Theorem, there exists a unique abelian extension $L_i / K$ all of whose ramified primes divide $f_i\ko_K$, such that
$\ker(\Phi^{L_i / K}_{f_i\ko_k} ) = P_{K,\Z}(f_i)$, i.e.~$\Gal(L_i/K) \cong C(\ko_{K,f_i})$. This extension is called the
\emph{ring class field} of $\ko_{K,f_i}$; at the level of ring class fields, we get induced a diagram of field extensions.

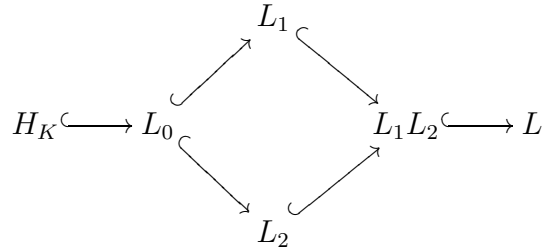
\[\xymatrix{
& &L_1  \ar@{^{(}->}[rd]  & &\\
H_K \ar@{^{(}->}[r] & L_0 \ar@{^{(}->}[ru] \ar@{^{(}->}[rd] & & L_1 L_2  \ar@{^{(}->}[r]& L \\
 & &L_2   \ar@{^{(}->}[ru]& &
}
\]

\noindent
By Galois theory, we get the following induced diagram of class groups.

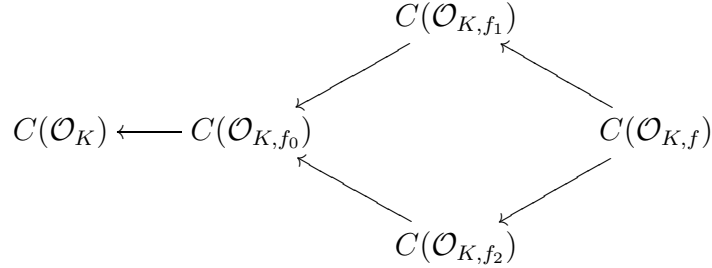
\[\xymatrix{
 & & C(\ko_{K,f_1}) \ar[ld] & \\
C(\ko_K)  & C(\ko_{K,f_0}) \ar[l] & & C(\ko_{K,f}) \ar[lu] \ar[ld] \\
 & & C(\ko_{K,f_2}) \ar[lu] &
}
\]
\noindent
In most cases the field $L$ is precisely the composite of $L_1$ and $L_2$, as it is stated in the following

\begin{prop}[Proposition 3.1 in \cite{allombert-bilu-pizarro14}]\label{key_cond}
 Assume all conditions above are satisfied. 
 \begin{enumerate}
  \item If $d_K \neq -3,-4$, then $L = L_1 L_2$.
  \item Assume $d_K \in \lbrace -3,-4 \rbrace$. 
    \begin{enumerate}
      \item If $f_1$ or $f_2$ is equal to 1, or $f_0 >1$, then $L = L_1 L_2$. 
      \item If $f_1,f_2 > 1$ and $f_0 =1$, then $L_1 L_2 \subsetneqq L$; moreover, the extension $L/L_1 L_2$ has degree 2 if $d_K =-4$,
      and degree 3 if $d_K = -3$.
    \end{enumerate}
 \end{enumerate}
\end{prop}

\subsection{Numbers represented by the principal form}

For two sets $S$ and $T$, we say that $S \,\dot{\subset}\, T$ if $S \subseteq T \cup \Sigma$, where $\Sigma$ is a finite set;
analogously, $S \doteq T$ means that both $S \,\dot{\subset}\, T$ and $T \,\dot{\subset}\, S$ hold. Suppose we are now given
a quadratic form $Q$; then, we can ask about the primes represented by $Q$, i.e.~about the set 
\[\kp_Q :=  \lbrace  \text{$p$ prime} \, \vert \, \text{$p$ is represented by $Q$} \rbrace.\]
It turns out that 
\[\kp_Q \doteq \Big\lbrace  \text{$p$ prime} \, \Big\vert \, \text{$p$ unramified in $K$,} \, 
\Big( \dfrac{L/K}{p} \Big)= \langle\sigma\rangle\Big\rbrace =: \hat{\kp}_Q ,\]
where $\langle \sigma \rangle$ is the conjugacy class of the element $\sigma \in \Gal(L/K)$ corresponding to the ideal associated
to the form $Q$, $K$ is the quadratic imaginary field of discriminant $\disc Q$, and $L$ is the ring class field of the order
$\ko$ of discriminant $\disc Q$. Notice that in case $Q= P$, the principal form, then $\hat{\kp}_P = \Spl(L/\Q)$, $\Spl(L/\Q)$
being the set of primes in $\Q$ that split completely in $L$. For later reference, we mention the following
\begin{lemma}[Exercise 8.14 in \cite{cox13}]\label{lemmasplit}
Let $L$ and $M$ be two finite extension of $K$, and let $\kp$ be a prime in $K$ that splits completely in both $L$ and $M$; then $\kp$
splits completely in the composite $LM$. Consequently, $\Spl(LM/K) =\Spl(L/K) \cap \Spl(M/K)$.
\end{lemma}

\section{Some admissible decompositions}\label{firstdec}
\subsection{Number of decompositions}
In \cite{ma11}, Ma solves the problem of finding the number of distinct decompositions of an abelian surface. The techniques
he employs are lattice theoretical, and the formulas strongly depend on the arithmetic of the transcendental lattice. However,
no explicit decomposition is exhibited for singular abelian surfaces.\newline

\noindent
We briefly recall Ma's results for an abelian surface $A$ with Picard number $\rho(A)=4$. Such a surface is necessarily the
product of two isogenous elliptic curves $E_1,E_2$ with complex multiplication. Following \cite{ma11}, if
$A \cong E_1 \times E_2$, then we say that the decomposition $(E_1,E_2)$ is \textit{admissible}. Two decompositions
$(E_1,E_2)$ and $(F_1,F_2)$ of $A$ are \emph{isomorphic} if $E_1 \cong F_1$ and $E_2 \cong F_2$, or $E_1 \cong F_2$ and
$E_2 \cong F_1$, and analogously, two decompositions $(E_1,E_2)$ and $(F_1,F_2)$ of $A$ are \emph{strictly isomorphic} if
$E_1 \cong F_1$ and $E_2 \cong F_2$. Let $\text{Dec}(A)$ be the set of isomorphism classes of decompositions of $A$, and
similarly let $\widetilde{\text{Dec}}(A)$ be the set of strict isomorphism classes of decompositions of $A$. Also, define
\[\delta(A) := \#{\text{Dec}}(A), \qquad \widetilde{\delta}(A):=\#\widetilde{\text{Dec}}(A).\]
To relate $\delta(A)$ and $\widetilde{\delta}(A)$, we consider the number of decompositions into the self-product of an
elliptic curve. To this end, we define
\[ \delta_0(A) := \# \big(\lbrace E \ \text{elliptic curve} : A \cong E \times E \rbrace / \cong \big),\]
and we have the obvious relation
\[\widetilde{\delta}(A) = 2 \delta(A) - \delta_0(A).\]

\noindent
For $n>1$, let $\tau(n)$ be the number prime factors of $n$, and set $\tau(1)=1$. Moreover, for a quadratic form $Q$,
let $g(Q)$ denote its genus, i.e.~the set of isometry classes of lattices isogenous to $Q$, and let $\widetilde{g}(Q)$ denote
its proper genus, i.e.~the set of isometry classes of oriented lattices isogenous to $Q$. Then we have the following result

\begin{thm}[Theorem 1.2, Theorem 1.3, Example 5.13 in \cite{ma11}]\label{thm_ma}
Let $A$ be an abelian surface of Picard number $\rho(A)=4$.
\begin{enumerate}
\item If $\rT(A)$ is not isometric to $\begin{pmatrix} 2n & 0 \\ 0 & 2n \end{pmatrix}$ or $\begin{pmatrix} 2n & n \\ n & 2n \end{pmatrix}$, $n>1$, one has
\[\delta(A) = \sum_{T \in g(\rT(A))} \# \big( O(A_T)/O(T) \big), \qquad \widetilde{\delta}(A)= 2^{-1} \cdot \#\widetilde{g}(\rT(A)) \cdot \# O(A_{\rT(A)}).\]
\item If $\rT(A) \cong \begin{pmatrix} 2n & 0 \\ 0 & 2n \end{pmatrix}$, then
\[\delta(A) = (2^{-4}+2^{-\tau(n)-3})\cdot \# O(A_{\rT(A)}), \qquad \widetilde{\delta}(A)=2 \delta(A).\]
\item If $\rT(A) \cong \begin{pmatrix} 2n & n \\ n & 2n \end{pmatrix}$, then
\[\delta(A) = \begin{cases} 3^{-2}\cdot (2^{-2}+2^{-\tau(n)}) \cdot \# O(A_{\rT(A)}) & n\text{ odd} \\ 3^{-2}\cdot (2^{-2}+2^{-\tau(2^{-1}n)}) \cdot \# O(A_{\rT(A)}) & n\text{ even} \end{cases}, \qquad \widetilde{\delta}(A)=2 \delta(A).\]
\item If $\rT(A)$ is either $\begin{pmatrix} 2 & 0 \\ 0 & 2 \end{pmatrix}$ or $\begin{pmatrix} 2 & 1 \\ 1 & 2 \end{pmatrix}$, then $\delta(A) = \delta_0(A) = \widetilde{\delta}(A)=1$.
\end{enumerate}
\end{thm}

\noindent
Under the assumption that $\rT(A)$ is primitive, Shioda and Mitani \cite[Theorem 4.7]{shioda-mitani74} proved a formula for
the number of decompositions which depended only on the structure of a certain class group. In \cite{ma11}, the aforementioned
formula is given the following interpretation

\begin{cor}[Corollary 5.11 in \cite{ma11}]\label{cor_ma2}
Let $A$ be a singular abelian surface having primitive transcendental lattice $\rT(A)$, and let $D:=- \det \rT(A) <0$. Then, $\widetilde{\delta}(A) = h(D)$.
\end{cor}

\noindent
Concerning the study of singular abelian surfaces with imprimitive transcendental lattice $\rT(A)$, Ma \cite{ma11} gives an
analogous formula for $\rT(A)$ not isometric to $\begin{pmatrix} 2n & 0 \\ 0 & 2n \end{pmatrix}$ or
$\begin{pmatrix} 2n & n \\ n & 2n \end{pmatrix}$, $n>1$.

\begin{cor}[Corollary 5.12 in \cite{ma11}]\label{cor_ma}
Let $A$ be a singular abelian surface having primitive transcendental lattice $\rT(A)$, and let $D:=- \det \rT(A) <0$. Let $A_n$ be the singular abelian surface of transcendental lattice $n \cdot \rT(A) := \rT(A)[n]$. If $\det \rT(A) \neq 3,4$, then $\widetilde{\delta}(A_n) = 2^{\tau(n)} \cdot h(n^2D)$.
\end{cor}

\subsection{Some explicit decompositions}
We now explain how to obtain a first type of decompositions, which are related to the primitive part of the transcendental lattice.
The idea comes from \cite{shioda-mitani74}, in particular their explicit description of the period map for singular abelian surfaces. Notice that Shioda and Mitani
gave a method to construct a singular abelian surface $A_Q$ of transcendental lattice $\rT(A)=Q$. Although the idea is the same, at
some point we will need some number theoretical statement from class group theory, crucial for our computations.\newline

\noindent
We now briefly explain where we got the idea from. In \cite{shioda-mitani74}, there is given a criterion to establish whether a
certain decomposition is admissible. If $A$ has transcendental lattice $Q=mQ_0$, with $Q_0$ primitive, let 
\[M_0:=Z + \tau(Q) \Z=\Z + \tau(Q_0) \Z,\] and let $f_0$ be the conductor of its CM ring 
\[\ko_{\Z + \tau(Q) \Z}=\Z + \frac{-b_0+\sqrt{D_0}}{2} \Z.\]

\begin{prop}[Proposition 4.5 in \cite{shioda-mitani74}]\label{thm_shioda-mitani}
Let $A_Q$ be the abelian surface associated to the quadratic form $Q$, and let $M_i$ be the module of conductor $f_i$ in $K=\Q(\tau(Q))$, $i=1,2$. Then $A \cong \C/M_1 \times \C/M_2$ if and only if $M_1 M_2 \sim M_0$, $(f_1,f_2)=f_0$ and $f_1 f_2 = mf_0^2$.
\end{prop}

\noindent
From this result, we now deduce some necessary conditions for a decomposition to be admissible. From the last two properties,
it follows that $\bar{f}_1 := f_1 /f_0$ and $\bar{f}_2 := f_2 /f_0$ are relatively prime, and hence we find a first upper bound
to the number of decompositions (absolutely not sharp, since we haven't used one of the conditions in
\cite[Proposition 4.5]{shioda-mitani74}). In fact, $f_0$ is uniquely determined by $Q$, and thus the only choice we have is
about $\bar{f}_1$ and $\bar{f}_2$, which have to satisfy $(\bar{f}_1,\bar{f}_2)=1$ and $n = \bar{f}_1 \bar{f}_2$. The number of
choices of pairs $(\bar{f_1},\bar{f_2})$ with $\gcd(\bar{f_1},\bar{f_2}) =1$ is indeed $2^{\tau(n)}$, and $f_1$ ($f_2$, respectively) determines univocally the order of conductor $f_1$ ($f_2$,
respectively), thus
\[\widetilde{\delta}(A) \leq \sum_{\substack{\bar{f}_1 \bar{f}_2 =n \\ (\bar{f}_1,\bar{f}_2)=1}} h(\ko_{K,f_1}) h(\ko_{K,f_2}) \leq 2^{\tau(n)} h(\ko_{K,n})^2.\]

\noindent
Notice that, if $Q=nQ_0$, then 
\[A_Q:=E_{\tau(Q_0)} \times E_{\tau(P)}=E_{\tau(Q_0)} \times E_{n\tau(P_0)},\]
 where $P$ is the principal form of discriminant $D=\disc Q$, and $P_0$ is the principal form of discriminant $D_0 = D/n^2$;
 also, $\rT(A) = n Q_0 = n (Q_0*P_0) $. This motivates the study of the transcendental lattice of $E_{\tau(Q_0)} \times E_{n\tau(Q_0')}$,
 where $Q_0, Q_0' \in C(D/n^2)$, which incidentally gives examples of decompositions coming from pairs $(\bar{f}_1,\bar{f}_2)=(1,n)$.
 More generally, we will be interested in abelian surfaces of the form 
 \[E_{s\tau(Q_0)} \times E_{t\tau(Q_0')},\]
where $st=n$. Under the assumption $\gcd(s,t)=1$, we are able to compute the transcendental lattice of this class of surfaces.

\begin{prop}\label{thm1}
Let $[Q]=s[Q_0],[Q']=t[Q_0']$ be such that $[Q_0],[Q_0'] \in {C}(D_0)$, for some $D_0<0$, and suppose $\gcd(s,t)=1$. Then 
\[\big[ \rT(E_{s\tau(Q_0)} \times E_{t\tau(Q_0')} \big] = st[Q_0 * Q_0'].\]
\end{prop}
\begin{proof}
Let
\[\tau_1 := s \dfrac{-b_0+\sqrt{D_0}}{2a_0}, \qquad \tau_2 := t \dfrac{-b_0'+\sqrt{D_0}}{2a_0'}.\]
Now, let $B_0$ be the element described in \cite[Lemma 3.2]{cox13}; by virtue of these relations (together with $\SL_2(\Z)$-invariance
of the $j$-invariant), we can replace $b_0$ and $b_0'$ by $B_0$ without changing the isomorphism classes of the elliptic curves.
Therefore, we can assume that 
\[\tau_1=s \dfrac{-B_0+\sqrt{D_0}}{2a_0}, \qquad \tau_2=s \dfrac{-B_0+\sqrt{D_0}}{2a_0'}.\]

\noindent
By \cite{shioda-mitani74}, 
\[p_A = u^{12}+\tau_2 u^{14}+\tau_1 u^{23} - \tau_1 \tau_2 u^{34},\]
and $\NS(A) = \ker (p_A)$. By picking an element
\[v = \sum_{1 \leq i < j \leq 4} A_{ij} u^{ij} \in \NS(A)_\Q,\]
and looking at its image under the period map, we see that
\begin{align*}
0 = p_A (v) = &\Big[A_{34} - \dfrac{tB_0}{2a_0'}A_{23}-\dfrac{sB_0}{2a_0}A_{14} - st\dfrac{D_0 + B_0^2}{4a_0a_0'}A_{12} \Big] +\\
&+ \sqrt{D_0} \Big[ \dfrac{t}{2a_0'}A_{23} + \dfrac{s}{2a_0}A_{14} + st \dfrac{B_0}{2a_0a_0'}A_{12}\Big].
\end{align*}
Solving the system of equations given by the pair of brackets, we get
\[A_{23} = -\dfrac{s}{a_0t}(a_0'A_{14}+tB_0A_{12}), \qquad A_{34} = st\dfrac{D_0 -B_0^2}{4a_0a_0'}A_{12},\]
which in turn give an explicit description of
\[\NS(A)_\Q = \Q \Big\langle u^{12} - \dfrac{sB_0}{a_0}u^{23} + std_0 u^{34} , u^{13} , u^{24} , u^{14} - \dfrac{sa_0'}{ta_0}u^{23}\Big\rangle,\]
with $d_0 := \dfrac{D_0 -B_0^2}{4a_0a_0'}$. 

\noindent
Now we want to compute $\rT(A) = \NS(A)^\perp$: let 
\[v=\sum_{1 \leq i < j \leq 4} A_{ij} u^{ij}  \in \NS(A)^\perp = \NS(A)_\Q^\perp,\]
and consider the relations coming by intersecting the generators of $\NS(A)_\Q$ with $v$.
\begin{align}
&A_{24} =0 \label{1}\\
&A_{13} =0 \label{2}\\
&A_{34} - \dfrac{sB_0}{a_0}A_{14} + std_0 A_{12} =0 \label{3}\\
&A_{23} -\dfrac{sa_0'}{ta_0}A_{14} =0 \label{4}
\end{align}
Equations (\ref{1}) and (\ref{2}) give clear conditions on the coefficients of $v$. Turning to (\ref{3}) and (\ref{4}), we can assume
that $(a_0,a_0')=(a_0,s)=1$ by \cite[Lemmata 2.3 and 2.25]{cox13}; furthermore, we can assume $t \nmid a_0'$. Under these hypotheses,
equation (\ref{4}) yields 
\[A_{14} = a_0 A_{14}' = a_0 t A_{14}'', \qquad A_{23} = sa_0'A_{14}'',\]  
and therefore we get
\[A_{34} = stB_0 A_{14}''-st d_0 A_{12}\]
from equation (\ref{3}). This leads to two generators for $\rT(A)$, namely
\[\rT(A) = \Z \Big\langle a_0t u^{14} + s a_0'u^{23} + st B_0 u^{34}, u^{12} - std_0 u^{34} \Big\rangle.\]
The intersection matrix of $\rT(A)$ is shown to be $st\begin{pmatrix} 2a_0a_0' & B_0 \\ B_0 & -2d_0 \end{pmatrix}$, and one can easily
check that the matrix is (positively) oriented.
\end{proof}

\noindent
If $A$ is a singular abelian surface of transcendental lattice $Q=nQ_0$, we get $2^{\tau(n)} h(D/n^2)$ decompositions of $A$, $D=-\det \rT(A)$. 

\begin{cor}\label{cor1}
If $\rT(A)=Q=nQ_0$, then we get decompositions of $A$ as $E_{s\tau(Q_0 * R_0)} \times E_{t \tau(R_0^{-1} * Q_0)}$, for
$R_0 \in C(\disc Q_0)$, $n=st$, $(s,t)=1$.
\end{cor}

\begin{ex}\label{ex1}
Let $A$ have transcendental lattice $6(1,0,3) \in \overline{C}(-432)$. Proposition \ref{thm_shioda-mitani} tell us that, in order
to find all decompositions of $A$, we must inspect the folowing orders.

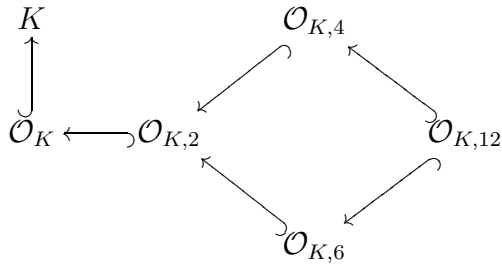
\[\xymatrix{
K & &\ko_{K,4} \ar@{^{(}->}[ld] & \\
\ko_K \ar@{^{(}->}[u] & \ko_{K,2} \ar@{^{(}->}[l] & & \ko_{K,12} \ar@{^{(}->}[lu] \ar@{^{(}->}[ld] \\
 & &\ko_{K,6} \ar@{^{(}->}[lu] &
}
\]

\noindent
According to Corollary \ref{cor_ma}, the number of strict decompositions is $\widetilde{\delta}(A) =24$, while Corollary \ref{cor1}
allows us to retrieve only $4$ of those. This happens because we only considered elliptic curves with CM by $\ko_{K,2}$ and $\ko_{K,12}$.
\qed
\end{ex}

\noindent
The example shows that, in order to being able to exhibit all decompositions, we need to compute the transcendental lattice of the
product of elliptic curves with CM by $\ko_{K,4}$ and $\ko_{K,3}$. The proof of Theorem \ref{thm1} suggests that this would be
possible if we were able to compose forms from different class groups.

\section{Explicit computation of transcendental lattices}\label{transcendental}

\subsection{Composition between different class groups}
The idea behind Dirichlet composition is that two forms $f(x,y)$ and $g(x,y)$ (having the same discriminant $D$) give rise to a
new form $F(x,y)$ (again of discriminant $D$) with the property
\[f(x,y) \cdot g(x,y) = F(B_1(x,y,z,w),B_2(x,y,z,w)),\]
for $B_i(x,y,z,w)\in \Z[xz,xw,yz,yw]$. In particular, the product of numbers represented
by $f(x,y)$ and $g(x,y)$ are represented by $F(x,y)$.\newline

\noindent
If $f(x,y)$ and $g(x,y)$ are not of the same discriminant, we can multiply them by a positive integer to obtain two new forms
having the same discriminant. Namely, given $Q_0 \in C(D_0)$ and $Q_0' \in C(D_0')$, with $D_0 = f_0^2 d_K$ and $D_0' = f_0'^2 d_K$,
set $f:=\text{lcm}(f_0,f_0')$. Then, putting
\[D:=f^2d_K, \qquad d:=f/f_0, \qquad d':=f/f_0',\]
the forms $Q:=d Q_0$ and $Q':=d'Q_0'$ have discriminant $D$; we get two classes in the \textit{extended class group}
\footnote{The name is actually misleading since this is not a group.}
\[\overline{C}(D):= \bigsqcup_{m \vert f(D)} m \cdot C(D/m^2),\]
namely $[Q]:=d[Q_0]$ and $[Q']:=d'[Q_0']$, and thus we can use Dirichlet composition after considering a suitable representative.
Since $d$ and $d'$ are coprime, $Q=(a,b,c)$ and $Q'=(a',b',c')$ can be assumed to have coprime leading coefficients, hence we do have
a composition: it is defined as usual (see \cite[Theorem 3.8]{cox13}), and it extends to elements of the (extended) class group. 

\begin{lemma}
Assume that $Q=(a,b,c)$ and $Q'=(a',b',c')$ are primitive, and suppose that 
\[n^2 \disc Q = m^2 \disc Q', \qquad \gcd(n,m)=1.\] 
Then, the form $(nQ)*(mQ')$ has primitivity index $nm$ (if the composition exists).
\end{lemma}

\noindent
This result follows from repeating the construction of the composition in this more general setup; the interested reader will find
a detailed account in \cite[Ch.~1, Sect.~3]{cox13}. In particular, going back to the case of $Q=dQ_0$ and $Q'=d'Q_0'$, we see that
$Q*Q'$ has primitivity index $dd'$. Also, 
\[D=\disc (Q*Q') = (dd')^2 \gcd(f_0,f_0')^2 d_K,\]
and therefore the primitive part of $Q*Q'$ comes from the order of conductor $\gcd(f_0,f_0')$.\\

\noindent
Under the 1:1 correspondence between the form class group $C(D)$ and the
ideal class group $C(\ko)$ (where $\ko$ is the unique order of discriminant $D$), we see that the form $(1,1,1) \in C(-3)$
corresponds to the ideal $[1,\frac{-1+\sqrt{-3}}{2}] \in C(\ko_K)$ ($K=\Q(\sqrt{-3})$). But also the form $(3,3,3)$ is sent
to the same ideal, and therefore we can freely lift a form to larger discriminant without changing the ideal class. This
suggests that the extended class group should be redefined as
\[\overline{C}(D):= \bigsqcup_{m \vert f(D)} C(D/m^2), \qquad \text{(drop $m$ in all factors)}\]
and consequently we define the \textit{extended ideal group}\footnote{Again, this is not a group!} as
\[\overline{C}(\ko):= \bigsqcup_{\ko \subseteq \ko' \subseteq \ko_K} C(\ko').\]
The bijection $C(D) \lra C(\ko)$ yields an analogous bijection $\overline{C}(D) \lra \overline{C}(\ko)$; this allows us to
work with ideal classes rather than forms.\newline

\noindent
Fix a quadratic imaginary field $K$, and let $\ko_1$ (resp. $\ko_2$) be the order of discriminant $D_1$ (resp. $D_2$) in
$\ko_K$; let $f_1$ (resp. $f_2$) be its conductor and set
\[f_0:=\gcd(f_1,f_2), \qquad \bar{f}_1:=f_1/f_0, \qquad \bar{f}_2:=f_2/f_0.\]
Composing forms of discriminants $D_1$ and $D_2$ gives forms of discriminant $D:=\lcm(D_1,D_2)$ ($f:=f(D)$), having
index of primitivity $d_1 d_2$, where $d_1 :=f/f_1$ and $d_2:=f/f_2$. Dropping the index we get a composition
\[C(D_1) \times C(D_2) \xrightarrow{\circledast} C(D_0),\]
where $D_0:=f_0^2 d_K$. More concretely, given $Q_1 \in C(D_1)$ and $Q_2 \in C(D_2)$, $Q_1 \circledast Q_2$ is the form in
$C(D_0)$ with the property that
\[d_1 d_2 [Q_1 \circledast Q_2] = [d_1 Q_1] * [d_2 Q_2].\]
On the level of ideal classes, we get
the usual multiplication between ideals
\[C(\ko_{K,f_1}) \times C(\ko_{K,f_2}) \xrightarrow{} C(\ko_{K,f_0}).\]
We now establish some elementary properties of $\circledast$ (and of $*$ in its original sense).
\begin{prop}
Let $Q_i \in C(D_i)$ ($i=0,1,2$), $R \in C(D)$, and let $P$ be the principal form of discriminant $D$. The composition $\circledast$ satisfies:
\begin{enumerate}
\item[($i$)] $[Q_0] \circledast [P] = [Q_0]$;
\item[($ii$)] $([Q_0]\circledast [R])\circledast [R]^{-1} = [Q_0]$; 
\item[($iii$)] $([Q_1] \circledast [R]) \circledast ([R]^{-1} \circledast [Q_2])= [Q_1] \circledast [Q_2]$;
\end{enumerate}
\end{prop}

\begin{proof}
Making use of the isomorphism between form class group and ideal class group, the proof follows easily from the corresponding properties for fractional ideals.
\end{proof}

\subsection{Explicit computation of transcendental lattices}
We now exhibit a formula for the transcendental lattice of a singular abelian surface, which is a product of two elliptic curves
$E_1 \in \ke ll (\ko_{K,f_1})$ and $E_2 \in \ke ll (\ko_{K,f_2})$. 

\begin{prop}\label{thm2}
Let $D_0= f_0^2 d_K$ and $D_0'= (f_0')^2 d_K$, where $d_K$ is the fundamental discriminant of a quadratic imaginary field $K$. Let $Q_0=(a_0,b_0,c_0) \in C(D_0)$ and $Q_0'=(a_0',b_0',c_0') \in C(D_0')$; if
\[f:=\lcm(f_0,f_0'), \quad d:=f/f_0, \quad d':=f/f_0',\]
consider the forms $Q:=dQ_0$ and $Q':=d'Q_0'$ of discriminant $D:=f^2d_K$. Set 
\[\tau := \tau(Q_0)=\frac{-b+\sqrt{D}}{2a}, \qquad \tau' := \tau(Q_0')=\frac{-b'+\sqrt{D}}{2a'},\]
where $a=da_0$, $b=db_0$ and $c=dc_0$, and let $E:=E_{\tau} \in \ke ll(\ko_{K,f_0})$ and $E':=E_{\tau '} \in \ke ll (\ko_{K,f_0'})$. Then 
\[[\rT(E \times E')] = [Q]*[Q'] = dd' [Q_0 \circledast Q_0'],\]
where $\circledast$ is the generalized Dirichlet composition.
\end{prop}
\begin{proof}
The proof is similar to the one of Theorem \ref{thm1}; we will just give an outline. By using Dirichlet composition, we can assume that
\[\tau := \frac{-B+\sqrt{D}}{2a}, \qquad \tau' := \frac{-B+\sqrt{D}}{2a'},\]
where $B$ is the integer coming into play because of the Dirichlet composition. Computations which are analogous to the ones above yield
\[\NS(A)_\Q = \Q \Big\langle u^{12}-\frac{B}{a}u^{23} + \frac{D-B^2}{4aa'}u^{34}, u^{14}-\frac{a'}{a}u^{23},u^{13},u^{24} \Big\rangle.\]
The transcendental lattice $\rT(A)$ is given by the conditions
\begin{align}
&A_{24}= A_{13}=0, \label{5}\\
&A_{34} - \dfrac{B}{a}A_{14} +  \frac{D-B^2}{4aa'} A_{12} =0, \label{6}\\
&A_{23} -\dfrac{a'}{a}A_{14} =0. \label{7}
\end{align}
Condition (7) gives
\[da_0 A_{23} = d' a_0' A_{14};\]
now we can assume that $(d,a_0')=1$ and then also that $(a_0,a')=1$. Under these assumptions, we see that 
\[A_{14} = a_0 A_{14}'=a_0dA_{14}'' \qquad \text{and} \qquad A_{23}=a_0'd'A_{14}'';\]
substituting in (6) yields
\[A_{34} = BA_{14}''+CA_{12},\]
and therefore we deduce
\[\rT(A) = \Z \Big\langle au^{14} + a'u^{23} + B u^{34}, u^{12} + C u^{34} \Big\rangle = \begin{pmatrix} 2aa' & B \\ B & 2 C \end{pmatrix}.\]
\end{proof} 

\begin{ex}[Example \ref{ex1} continued]
We can now get all decompositions of the abelian surface $A$ having transcendental lattice $\rT(A) = 6(1,0,3)$. In fact,
since $h(O_{K,2})=1$, any pair of elliptic curves $(E_1,E_2)$ with $E_1 \in \ke ll (\ko_{K,4})$ and $E_2 \in \ke ll (\ko_{K,3})$
gives a decomposition; also, we can use pairs $(E_1,E_2)$ with $E_1 \in \ke ll (\ko_{K,2})$ and $E_2 \in \ke ll (\ko_{K,12})$.
It is easy to verify that we get exactly 24 strict decompositions.
\end{ex}

\subsection{New candidate decompositions}
\noindent
As a consequence of Proposition \ref{thm2}, we get a new family of decompositions of a given abelian surface $A$ of transcendental
lattice $Q= n Q_0$. Indeed, the group $C(D)$ acts on the class groups $C(D_1)$ and $C(D_2)$ by $\circledast$, and therefore,
once we are given a decomposition $A = E_{\tau(Q_1)} \times E_{\tau(Q_2)}$, we get new ones by taking
\[E_{\tau([Q_1] \circledast [R])} \times E_{\tau([Q_2] \circledast [R]^{-1})}, \qquad [R] \in C(D).\]
Notice that we can always cook up such a decomposition: consider the forms $Q_0$ and $P_0$ (the latter being the
principal form) of discriminant $D_0$, and if $s,t \in \Z$ are coprime nonnegative integers consider the abelian surface
$E_{s\tau(Q_0)} \times E_{t\tau(P_0)}$ as in Theorem \ref{thm1}. Then, it gives indeed a decomposition of $A$; now notice
that $s \tau(Q_0)$ corresponds to the form
\[ax^2+(bs)xy+(cs^2)y^2,\]
which is primitive in $C(s^2D)$, and similar considerations
hold for $t \tau(P_0)$.\\

\noindent
Now that we have these families of decompositions, is there a way of getting them all? Namely, to what extent does the action of
$C(D)$ on $C(D_1)$ and $C(D_2)$ give a description of the possible decompositions?

\section{Decompositions in the cases $K\neq\Q(i),\Q(\sqrt{-3})$}\label{alldec}

\subsection{Action of a class group on class groups of smaller discriminant}
Recall that if $D_0 \vert D$, the class group $C(D)$ acts on $C(D_0)$. Therefore, we might ask whether the action is transitive.
Notice that a form $Q_0 \in C(D_0)$ can be lifted to a primitive form $Q \in C(D)$ in such a way that $Q \circledast P_0 = Q_0$. 

\begin{lemma}
For every form $Q_0 \in C(D_0)$ there exists a form $Q \in C(D)$ which is the lift of $Q_0$ in the following sense: $Q \circledast P_0 = Q_0$.
\end{lemma}

\begin{proof}
If $Q_0=[a_0,b_0,c_0]$ is represented by the ideal $[a_0, \frac{-b_0 + \sqrt{D_0}}{2}]$, then $Q$ correspond to the ideal $[a_0, \frac{-db_0 + \sqrt{D}}{2}]$; also $dP_0=[d,\frac{-dp_0 + \sqrt{D}}{2}]$, where $p_0=0,1$ according to the parity of the discriminant $D$. It follows that
\[[Q]*[dP_0] = [a_0d, \Delta] = [a_0,\Delta/d] = [Q],\]
where $\Delta = \frac{-B + \sqrt{D}}{2}$, and $B$ is the usual key integer in the Dirichlet composition.
\end{proof}

\noindent
As a consequence, we have the following

\begin{cor}
The action of $C(D)$ on $C(D_0)$ is transitive.
\end{cor}

\noindent
This means that the factors of the decompositions
\[E_{\tau([Q_1] \circledast [R])} \times E_{\tau([Q_2] \circledast [R]^{-1})}, \qquad [R] \in C(D)\]
cover the whole class groups $C(D_1)$ and $C(D_2)$. However, we do not know whether we get distinct decompositions under
this action. Of course, if this were the case, then we would obtain the whole set of decompositions, matching Ma's formula.

\subsection{Distinct decompositions}
We now come to the issue of whether the set of decompositions we get with the above technique is complete or not. To do so,
let us assume $Q_1 \in C(D_1)$, $Q_2 \in C(D_2)$ and $R,S \in C(D)$. Moreover, suppose that 
\[[Q_1] \circledast [R] = [Q_1] \circledast [S], \qquad [Q_2] \circledast [R]^{-1} = [Q_2] \circledast [S]^{-1},\]
which is the case of a decomposition being realized by two elements $R,S \in C(D)$. This is equivalent to the existence of
an element $U \in C(D)$ such that
\[U \circledast Q_1 = Q_1, \qquad U \circledast Q_2 = Q_2.\]
So we are to understand the elements $U \in C(D)$ that fix $Q_i$, $i=1,2$.\newline

\noindent
Let $C(D)$ act on $C(D_0)$, and let $U$ be an element fixing some $Q_0 \in C(D_0)$; notice that $U$ would actually fix the
whole class groups $C(D_0)$. We call the group of such $U$'s the \emph{stabilizer} of $C(D_0)$ in $C(D)$, and it will be
denoted by $\Stab C(D_0)$; clearly, its order is $h(D) / h(D_0)$. In the situation of interest  to us, we can consider the
intersection $\Stab C(D_1) \cap \Stab C(D_2)$: it describes the elements in $C(D)$ that represent an obstruction to having
the full set of decompositions by twisting by $C(D)$ the factors of a given decomposition. \newline

\noindent
It may occur that $\Stab C(D_1) \cap \Stab C(D_2)$ is trivial, for instance when the class number or the index of primitivity
are small enough, or even when the orders of the stabilizers are powers of different primes. However, it is not difficult to
come up with an example of this not being the case when $K = \Q(i)$ or $K = \Q(\sqrt{-3})$.

\begin{ex}
Let $Q = 30 \cdot (1,0,3) \in \overline{C}(D)$, where $D = -60^2 3$. If $A$ is the singular abelian surface with $T(A)=Q$, then we know by Corollary \ref{thm_ma} that $\widetilde{\delta}(A) = 288$. There are four classes of products we have to consider.
\begin{enumerate}
    \item $\ke ll (\ko_{K,2}) \times \ke ll (\ko_{K,60})$: we already noticed above that this class gives $2 h(\ko_{K,60}) = 72$ distinct strict decompositions.
    \item $\ke ll (\ko_{K,4}) \times \ke ll (\ko_{K,30})$: we need to estimate the order of the stabilizers. We have (by means of the class number formula)
    \[\# \Stab C(\ko_{K,4}) = h(\ko_{K,60})/h(\ko_{K,4}) = 18,\]
    \[\# \Stab C(\ko_{K,30}) = h(\ko_{K,60})/h(\ko_{K,30}) = 2. \]
    It follows that $\# (\Stab C(\ko_{K,4}) \cap \Stab C(\ko_{K,30})) \leq 2$, and thus we get at least $h(\ko_{K,60})=36$ distinct strict decompositions.
    \item $\ke ll (\ko_{K,6}) \times \ke ll (\ko_{K,20})$: by using the class number formula, we see that
    \[\# \Stab C(\ko_{K,6}) = h(\ko_{K,60})/h(\ko_{K,6}) = 12,\]
    \[\# \Stab C(\ko_{K,20}) = h(\ko_{K,60})/h(\ko_{K,20}) = 3; \]
    thus $\# (\Stab C(\ko_{K,6}) \cap \Stab C(\ko_{K,20})) \leq 3$, and we get at least $2h(\ko_{K,60})/3=24$ distinct strict decompositions.
    \item $\ke ll (\ko_{K,10}) \times \ke ll (\ko_{K,10})$: by using the class number formula, we see that
    \[\# \Stab C(\ko_{K,10}) = h(\ko_{K,60})/h(\ko_{K,10}) = 6,\]
    \[\# \Stab C(\ko_{K,12}) = h(\ko_{K,60})/h(\ko_{K,12}) = 6; \]
    thus $\# (\Stab C(\ko_{K,10}) \cap \Stab C(\ko_{K,12})) \leq 6$, and we get at least $2h(\ko_{K,60})/6=12$ distinct strict decompositions.
\end{enumerate}
In total, we have obtained at least 144 distinct decompositions out of 288.
\end{ex}

\noindent
The question we would like to answer is the following
\begin{question}\label{question}
Is $\Stab C(\ko_{K,f_1}) \cap \Stab C(\ko_{K,f_2})$
is always trivial, or at least when $K \neq \Q(i), \Q(\sqrt{-3})$?
\end{question}
\noindent
Indeed, this will be the case, as we are going to show in the following.

\subsection{Answer to the question}

The key ingredient is the fact that the principal form represents all but finitely many unramified primes which split completely in the ring class field. Let $P_i$ be the principal form of the order $\ko_{K,f_i}$, $i= 1,2,\emptyset$. Also, let $L_i$ be the ring class field
of the order $\ko_{K,f_i}$, $i= 1,2,\emptyset$.

\begin{lemma}\label{split}
Let $\kp_{P_i}$ be the set of primes of represented by $P_i$, for $i= 1,2,\empty$. Then, $\kp_{P} \doteq \kp_{P_1} \cap \kp_{P_2}$.
\end{lemma}

\begin{proof}
By using Lemma \ref{lemmasplit}, we see that
\[{\kp}_P \doteq \Spl(L/K) = \Spl(L_1/K) \cap \Spl(L_2/K) \doteq \kp_{P_1} \cap \kp_{P_2}.\]
\end{proof}

\noindent
By the \v{C}ebotarev Density Theorem, we can reason with the set $\kp_{P_i}$ rather than $\Spl(L_i/K)$, for
$i= 1,2,\emptyset$: in fact, they both have positive Dirichlet density (thus they are infinite), and they are
the same up to a finite set (which has Dirichlet density 0). 

\begin{prop}\label{char_princ_form}
The principal form is characterized by representing almost all primes that split completely in the ring class field.
\end{prop}

\begin{proof}
Suppose $Q$ is a form such that $\kp_Q \doteq \Spl(L/K)$. Then, we would have
\[\Big\lbrace \text{$p$ prime} \, \Big\vert \, \text{$p$ unramified}, \, \Big( \dfrac{L/K}{p} \Big)= \langle\sigma\rangle \Big\rbrace  \doteq \Big\lbrace \text{$p$ prime} \, \Big\vert \, \Big( \dfrac{L/K}{p} \Big)= \langle 1 \rangle \Big\rbrace, \]
and since both sets have infinitely many elements it must necessarily be $\sigma=1 \in \Gal(L/K)$, which corresponds to the class of the principal form. Since equivalent forms represent the same numbers, we are done.
\end{proof}

\noindent
We can now answer Question \ref{question}:

\begin{thm}\label{thmthm}
Unless $d_K \in \lbrace -3,-4\rbrace$, $f_1,f_2 >1$ and $f_0=1$, we have
\[\Stab C(\ko_{K,f_1}) \cap \Stab C(\ko_{K,f_2}) = (0).\]
\end{thm}

\begin{proof}
Let $[Q] \in \Stab C(D_1) \cap \Stab C(D_2)$, i.e. $[Q]$ is such that
\[Q \circledast P_1 = P_1, \qquad Q \circledast P_2 = P_2.\] 
Now, for $i=1,2$, the primes represented by $P_i$ are, up to a finite set, those $p$ that split completely in the ring
class field $L_i$. In the same fashion, the primes represented by $Q$ are, up to a finite set, the ones splitting
completely in the ring class field $L$. Notice that, by the assumption, it follows that all primes represented by $Q$
are also represented by $P_1$ and $P_2$. Moreover, by Proposition \ref{key_cond}, $L = L_1 L_2$, and Lemma \ref{split}
and Proposition \ref{char_princ_form} imply that $Q$ is in fact the principal form.
\end{proof}

\noindent
Now, let us recall that to an element $[Q] \in C(\ko)$ in a class group we can associate an elliptic curve by setting $E_Q := E_{\tau(Q)}$. In light of this, we can rephrase the previous result as follows.
\begin{thm}\label{mainthm}
Unless $d_K \in \lbrace -3,-4\rbrace$, $f_1,f_2 >1$ and $f_0=1$, the group $C(\ko_{K,f})$ spans all the possible decompositions of $A$ into the products of elliptic curves with classes in $C(\ko_{K,f_1})$ and $C(\ko_{K,f_2})$.
\end{thm}

\noindent
As a consequence, we obtain the classification theorem for decompositions of singular abelian surfaces in case the transcendental lattice is not a multiple of $\begin{pmatrix} 2 & 0 \\ 0 & 2 \end{pmatrix}$ or $\begin{pmatrix} 2 & 1 \\ 1 & 2 \end{pmatrix}$. 

\begin{thm}\label{classthm}
Let $A$ be a singular abelian surface having transcendental lattice $Q = nQ_0$, with $Q_0$ neither $\begin{pmatrix} 2 & 0 \\ 0 & 2 \end{pmatrix}$ nor $\begin{pmatrix} 2 & 1 \\ 1 & 2 \end{pmatrix}$; let $D=f^2 d_K =\disc Q$, $D_0 = f_0^2 d_K = \disc Q_0$ and consider all pairs $(f_1,f_2)$ of positive integer such that
\[\gcd(f_1,f_2) = f_0 \qquad \text{and} \qquad f_1 f_2 = nf_0^2.\]
Then $A$ decomposes into the product of two mutually isogenous elliptic curves with complex multiplication according to one of the following possibilities:

\begin{enumerate}
    \item $A \cong E_{\tau([Q_0] \circledast [R])} \times E_{\tau([P] \circledast [R]^{-1})}$, for $[R] \in C(D)$;
    \item for any choice of a pair $(f_1,f_2)$ such that
\[\gcd(f_1,f_2)=1 \qquad \text{and} \qquad f_1 f_2 = n f_0,\]
take a form in the class $[Q_0]$ that lifts to a primitive form $Q_1$ of discriminant $D_1$, and similarly lift a form in $[P_0]$ to a primitive form $Q_2$ of discriminant $D_2$; then, $A \cong E_{\tau([Q_1] \circledast [R])} \times E_{\tau([Q_2] \circledast [R]^{-1})}$, for $[R] \in C(D)$.
\end{enumerate}
\end{thm}
\begin{proof}[Proof of Theorem \ref{classthm}]
The result follows by applying Theorem \ref{thmthm} to each pair $(f_1,f_2)$ such that
\[f_1 f_2 = n f_0^2 \qquad \text{and} \qquad \gcd(f_1,f_2)=f_0,\]
and by noticing that the number of decompositions we get matches with Ma's original formula.
\end{proof}

\subsection{Alternative proof of Ma's formula}

The classification of decompositions of singular abelian surfaces has been obtained by producing enough distinct decompositions to match Ma's formula. However, our construction incidentally provides the reader with an alternative and simpler proof of the same result, at least in the cases covered by the classification theorem (Theorem \ref{thmthm}).\\

\noindent
Let $\Sigma^{\rm Ab}(D,n)$ be the space of singular abelian surfaces of discriminant $D$ and primitivity index $n$, i.e.~the space of surfaces $A$ such that $\rT(A) = n Q_0$, for a primitive form $Q_0$, and $\disc \rT(A) = D$. If we consider all the elements of $\Sigma^{\rm Ab}(D,n)$ at once, we have constructed a total of
\[2^{\tau(n)}h(\ko_{K,f_0})h(\ko_{K,f})\]
distinct product surfaces. However, the number of distinct product surfaces within $\Sigma^{\rm Ab}(D,n)$ is
\[\sum_{A \in \Sigma^{\rm Ab}(D,n)} \widetilde{\delta(A)}=\sum_{\substack{(f_1,f_2)=f_0 \\ f_1 f_2 =nf_0^2 }}h(\ko_{K,f_1})h(\ko_{K,f_2}).\]
We now show that these two numbers are indeed the same; therefore, this result yields a different proof of Ma's formula in most cases.

\begin{prop}
Unless $d_K \in \lbrace -3,-4\rbrace$ and $f_0 =1$, we have
\[2^{\tau(n)} h(\ko_{K,f})h(\ko_{K,f_0}) = \sum_{\substack{(f_1,f_2)=f_0 \\ f_1 f_2 =nf_0^2 }}h(\ko_{K,f_1})h(\ko_{K,f_2}).\]
\end{prop}

\begin{proof}
We notice that
\begin{align*}
\sum_{\substack{(f_1,f_2)=f_0 \\ f_1 f_2 =nf_0^2 }}h(\ko_{K,f_1})h(\ko_{K,f_2}) &= 2\sum_{\substack{(f_1,f_2)=f_0 \\ f_1 f_2 =nf_0^2 \\ f_1 < f_2}}h(\ko_{K,f_1})h(\ko_{K,f_2})\\
&=2\sum_{\substack{(f_1,f_2)=f_0 \\ f_1 f_2 =nf_0^2 \\ f_0 \neq f_1 < f_2}}h(\ko_{K,f_1})h(\ko_{K,f_2}) + 2h(\ko_{K,f_0})h(\ko_{K,f}),
\end{align*} 
and so it is enough to prove that
\[(2^{\tau(n)-1}-1)h(\ko_{K,f})=\sum_{\substack{(f_1,f_2)=f_0 \\ f_1 f_2 =nf_0^2 \\ 1 \neq f_1 < f_2}} \frac{h(\ko_{K,f_1})h(\ko_{K,f_2})}{h(\ko_{K,f_0})}; \]
since the number of summands on the right-hand side is precisely $2^{\tau(n)-1}-1$, we are left to prove that 
\[\frac{h(\ko_{K,f_1})h(\ko_{K,f_2})}{h(\ko_{K,f_0})}=h(\ko_{K,f}).\]
But this is a consequence of the class number formula and some easy arithmetic: in fact, by the assumptions, 
\[[\ko_K^\times : \ko_{K,f_i}^\times] = \# \ko_K^\times, \qquad i=0,1,2.\]
Setting
\[\Pi_i:= \prod_{p \vert f_i} \Bigg(1- \Bigg(\dfrac{d_K}{p}\Bigg)\dfrac{1}{p} \Bigg), \qquad i=0,1,2,\emptyset\]
Theorem \ref{clnumbfor} yields
\begin{align*}
\dfrac{h(\ko_{K,f_1})h(\ko_{K,f_2})}{h(\ko_{K,f_0})} &= \dfrac{h(\ko_K) f_1 f_2 \Pi_1 \Pi_2}{f_0  \Pi_0 \#\ko_K^\times} = \dfrac{h(\ko_K) f}{\# \ko_K^\times} \cdot \dfrac{\Pi_1 \Pi_2}{\Pi_0}, 
\end{align*}
since $f=nf_0$ and $f_1 f_2 = nf_0^2 = f f_0$. However, it is straightforward to see that
\[\dfrac{\Pi_1 \Pi_2}{\Pi_0} = \Pi,\]
and therefore the proof is complete.
\end{proof}

\section{Decompositions in the remaining cases}\label{remainingcases}

\subsection{Number of decompositions and classification}
Unfortunately, the techniques employed thus far cannot be employed when $K=\Q(i)$ or $K=\Q(\sqrt{-3})$, as we might have nontrivial elements in $\text{Stab}\, C(\ko_{K,f_1}) \cap \text{Stab}\, C(\ko_{K,f_2})$. However, we are still able to completely solve the classification problem. First of all, we would like to point out that a formula for the number of decompositions in these two cases can be obtained just by counting, as the class number of $\ko_K$ is one.\\

\noindent
Assume first that $K=\Q(i)$ and that we are given a singular abelian surface $A$ with transcendental lattice $\begin{pmatrix} 2n & 0 \\ 0 & 2n \end{pmatrix}$; then, we have the following diagram,

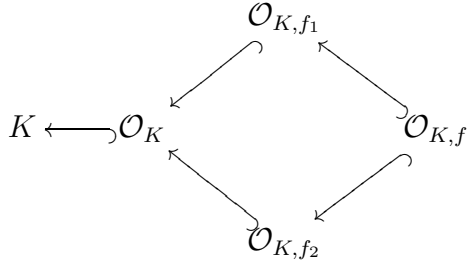
\[\xymatrix{
 & &\ko_{K,f_1} \ar@{^{(}->}[ld] & \\
K & \ko_{K} \ar@{^{(}->}[l] & & \ko_{K,f} \ar@{^{(}->}[lu] \ar@{^{(}->}[ld] \\
 & &\ko_{K,f_2} \ar@{^{(}->}[lu] &
}
\]
\noindent
where $f_1$ and $f_2$ are relatively prime, and actually $f=n$. 

\begin{thm}\label{disc-4}
Given a singular abelian surface $A$ with transcendental lattice $\begin{pmatrix} 2n & 0 \\ 0 & 2n \end{pmatrix}$, the number of decompositions of $A$ into the product of two mutually isogenous elliptic curves with complex multiplication (up to isomorphism of the factors) is 
\[\tilde{\delta}(A) = (1 + 2^{\tau(n)-1}) h(\ko_{K,n}).\]
The surface $A$ is isomorphic to any of the products $(E_1,E_2)$, where $[E_i] \in \ke ll (\ko_{K,f_i})$  ($i=1,2$), $\gcd(f_1,f_2) =1$ and $f_1 f_2=n$.
\end{thm}
\begin{proof}
Since $h(\ko_K)=1$, we see that

\begin{align*}
\tilde{\delta}(A) &= \sum_{\substack{(f_1,f_2)=1 \\ f_1 f_2 =n }}h(\ko_{K,f_1})h(\ko_{K,f_2}) = 2\sum_{\substack{(f_1,f_2)=1 \\ f_1 f_2 =n \\ f_1 < f_2}}h(\ko_{K,f_1})h(\ko_{K,f_2})\\
&=2\sum_{\substack{(f_1,f_2)=1 \\ f_1 f_2 =n \\ 1 \neq f_1 < f_2}}h(\ko_{K,f_1})h(\ko_{K,f_2}) + 2h(\ko_{K,n}).
\end{align*}

\noindent
Following the notation from earlier, since $\#\ko_K^\times = 4$, the class number formula implies that $h(\ko_{K,f_i}) = f_i \Pi_i /2$. Therefore, $h(\ko_{K,f_1})h(\ko_{K,f_2}) = n \Pi /4$, and thus
 
\begin{align*}
\tilde{\delta}(A) &= (2^{\tau(n)-1}-1)n \Pi /2 + n \Pi \\
&= \frac{1}{2}n \Pi (1 + 2^{\tau(n)-1}) = (1 + 2^{\tau(n)-1}) h(\ko_{K,n}).
\end{align*}

\noindent
So we are able to exhibit a formula for the number of decompositions of such a singular abelian surface. Also, the classification problem is solved, as we can just take all pairs $(E_1,E_2)$ fitting in the diagram above. 
\end{proof}

\noindent
The case $K=\Q(\sqrt{-3})$ is analogous, and thus the proof of Theorem \ref{disc-3} is the same, except for the fact that $\# \ko_{K} ^\times = 6$. 

\begin{thm}\label{disc-3}
Given a singular abelian surface $A$ with transcendental lattice $\begin{pmatrix} 2n & n \\ n & 2n \end{pmatrix}$, the number of decompositions of $A$ into the product of two mutually isogenous elliptic curves with complex multiplication (up to isomorphism of the factors) is
\[\tilde{\delta}(A) = \frac{2}{3}(2 + 2^{\tau(n)-1}) h(\ko_{K,n}).\]
The surface $A$ is isomorphic to any of the products $(E_1,E_2)$, where $[E_i] \in \ke ll (\ko_{K,f_i})$  ($i=1,2$), $\gcd(f_1,f_2) =1$ and $f_1 f_2=n$.
\end{thm}

\noindent
Although the counting and the classification problems are completely solved, it is not clear whether the action of the class group $C(\ko_{K,f})$ spans the set of decompositions of $A$.

\subsection{Application: Shioda-Inose models of singular K3 surfaces}

Let $X$ be a singular K3 surface, and let $\rT(X)$ denote its transcendental lattice. By results of Shioda and Inose \cite{shioda-inose77}, there exists a singular abelian surface $A=E_1 \times E_2$ such that $\rT(A) = \rT(X)$; moreover, there is a model of $X$ which is given in terms of the $j$-invariants of $E_1$ and $E_2$. In \cite{schuett07}, the author gives a very nice model for such a K3 surface, namely $X$ has the following model as an elliptic fibration
\[ X: \qquad y^2 = x^3 -3ABt^4x+ABt^5(Bt^2-2Bt+1),\]
where $A=j_1 j_2$ and $B=(1-j_1)(1-j_2)$, $j_k$ being the $j$-invariant of $E_k$ ($k=1,2$). It follows that our classification of the decompositions of a singular abelian surface gives all the possible Shioda-Inose models of $X$, i.e. all the possible models of $X$ which are realizable via a Shioda-Inose structure.

\subsection{Application: fields of moduli of singular K3 surfaces}

We now provide an application of our results to the theory of singular K3 surfaces. It is well-known that every singular abelian surface yields a singular K3 surface by means of a Shioda-Inose structure \cite{shioda-inose77}. We are interested in some arithmetic invariant of the latter, namely the field of moduli. We recall that $M$ is the (absolute) \textit{field of moduli} of a variety $X$ over a number field if for all automorphisms $\sigma \in \Aut(\C/\Q)$,
\[X^\sigma \in [X] \Longleftrightarrow \text{$\sigma$ acts trivially on $M$},\]
where by $[X]$ we denote the isomorphism class of $X$.
The field of moduli exists and is unique, since by Galois theory it is equivalently defined as the fixed field of the group
\[G:=\lbrace \sigma \in \Aut(\C/\Q) \, \vert \, X^\sigma \in [X] \rbrace.\]

\noindent
Generalizing a previous result of Shimada \cite{shimada09}, Sch\"{u}tt was able to prove the following result

\begin{thm}[Theorem 5.2 in \cite{schuett07}]
Let $X$ be a singular K3 surface, and let $\rT(X)$ be its transcendental lattice. Assume that $X$ is defined over a Galois extension $L/K$, where $K=\Q(\disc \rT(A))$. Then, the action of the Galois group $\Gal(L/K)$ spans the genus of $\rT(X)$, i.e.

\[(\text{genus of $\rT(X)$}) = \lbrace [\rT(X^\sigma)] \, \vert \, \sigma \in \Gal(L/K) \rbrace.\]
\end{thm}

\noindent
Set $L:=H(\disc \rT(X))$ in the statement above, where $H(D)$ denotes the ring class field of the order in $K$ of discriminant $D$, for $D < 0$. Class field theory tells tells us that
\[\Gal(L/\Q) \cong \Gal(L/K) \rtimes \Gal(K/\Q),\]
where $\Gal(K/\Q)$ accounts for the complex conjugation (for a reference, see \cite[Ch.~9]{cox13}). But complex conjugation has the effect of sending a singular K3 surface of transcendental lattice $\begin{pmatrix} 2a & b \\ b & 2c \end{pmatrix}$ to the singular K3 surface with transcendental lattice $\begin{pmatrix} 2a & -b \\ -b & 2c \end{pmatrix}$, so it acts as inversion on the corresponding class group (see \cite{shioda-mitani74} and \cite{schuett07}). By observing that a form and its inverse lie in the same genus, we conclude that
\begin{align*}
(\text{genus of $\rT(X)$}) =& \lbrace [\rT(X^\sigma)] \, \vert \, \sigma \in \Gal(L/K) \rbrace =\\
=& \lbrace [\rT(X^\sigma)] \, \vert \, \sigma \in \Gal(L/\Q) \rbrace.
\end{align*}
This suggests a connection between the field of moduli of a singular K3 surface and the genus of its transcendental lattice. In fact, the problem of characterizing the field of moduli of singular K3 surfaces is dealt with in a paper of the author \cite{laface16}.\\

\noindent
The classification of decompositions of a singular abelian surface allows us to tell something more about the field of moduli of $X$. Recall that $M$ is contained in the intersection of all possible fields of definition. Then, by means of a Shioda-Inose structure starting from a suitable singular abelian surface $E_1 \times E_2$, $X$ admits a model over $\Q(j_1 j_2, j_1 +j_2)$ by a result of Sch\"utt \cite{schuett07}. Therefore, considering all admissible pairs $(E_1,E_2)$ such that $\rT(E_1 \times E_2) = \rT(X)$, we see that
\[M \subseteq \bigcap_{\text{$X$ defined over $L$}} L \subseteq \bigcap_{\text{$j_1,j_2$ as above}} \Q(j_1 j_2, j_1 +j_2).\]
We deduce a slightly clearer picture of what $M$ looks like, as we know where it has to sit as an extension of $\Q$. Namely, $M$ lies in right-hand side above, which is not hard to describe theoretically. In practice, describing it is a hard task, as this involves the computation of several $j$-invariants.

\subsection{Open problems}
The present treatment deals with decompositions of singular abelian surfaces, but one might want to investigate the possible decompositions in the case of singular abelian varieties of higher dimension. It was proven by Katsura \cite{katsura75} that such a variety is isomorphic to the product of mutually isogenous elliptic curves with CM.

\begin{problem}
Given a singular abelian variety $A$, how many decompositions of $A$ into the product of mutually isogenous elliptic curves are there? What are the possible decompositions?
\end{problem}

\bibliographystyle{plain}
\bibliography{bib}{}

\begin{thebibliography}{10}

\bibitem{allombert-bilu-pizarro14}
B.~{Allombert}, Y.~{Bilu}, and A.~{Pizarro-Madariaga}.
\newblock {CM-Points on Straight Lines}.
\newblock {\em \rm{arXiv:1406.1274}}, June 2014.

\bibitem{cox13}
D.~A. Cox.
\newblock {\em Primes of the form {$x^2 + ny^2$}}.
\newblock Pure and Applied Mathematics (Hoboken). John Wiley \& Sons, Inc.,
  Hoboken, NJ, second edition, 2013.
\newblock Fermat, class field theory, and complex multiplication.

\bibitem{katsura75}
T.~Katsura.
\newblock On the structure of singular abelian varieties.
\newblock {\em Proc. Japan Acad.}, 51(4):224--228, 1975.

\bibitem{laface16}
R.~{Laface}.
\newblock {The field of moduli of singular abelian and K3 surfaces}.
\newblock {\em ArXiv e-prints}, January 2016.

\bibitem{ma11}
S.~Ma.
\newblock Decompositions of an {A}belian surface and quadratic forms.
\newblock {\em Ann. Inst. Fourier (Grenoble)}, 61(2):717--743, 2011.

\bibitem{morrison84}
D.R. Morrison.
\newblock {On K3 surfaces with large Picard number}.
\newblock {\em Inventiones mathematicae}, 75(1):105--121, 1984.

\bibitem{schuett07}
M.~Sch{\"u}tt.
\newblock Fields of definition of singular {$K3$} surfaces.
\newblock {\em Commun. Number Theory Phys.}, 1(2):307--321, 2007.

\bibitem{shimada09}
I.~Shimada.
\newblock Transcendental lattices and supersingular reduction lattices of a
  singular {$K3$} surface.
\newblock {\em Trans. Amer. Math. Soc.}, 361(2):909--949, 2009.

\bibitem{shioda-inose77}
T.~Shioda and H.~Inose.
\newblock On singular {$K3$} surfaces.
\newblock In {\em Complex analysis and algebraic geometry}, pages 119--136.
  Iwanami Shoten, Tokyo, 1977.

\bibitem{shioda-mitani74}
T.~Shioda and N.~Mitani.
\newblock Singular abelian surfaces and binary quadratic forms.
\newblock In {\em Classification of algebraic varieties and compact complex
  manifolds}, pages 259--287. Lecture Notes in Math., Vol. 412. Springer,
  Berlin, 1974.

\end{thebibliography}
\end{document}